\documentclass{article}

\usepackage[utf8]{inputenc}
\usepackage[italian,english]{babel}

\usepackage{geometry}								
\usepackage[linktoc=all,pdfusetitle,pdfauthor={Dario Ascari and Francesco Milizia},pdfborder={0 0 .5}]{hyperref}	
\usepackage{tikz}				
\usepackage{tikz-cd}             
\usepackage{enumitem}			
\usepackage{float}				

\usepackage{amsmath}
\usepackage{amssymb}	
\usepackage{mathrsfs}	
\usepackage{amsthm}	
\usepackage{bm}

\theoremstyle{plain}
\newtheorem{mythm}{Theorem}[section]
\newtheorem{myprop}[mythm]{Proposition}
\newtheorem{mylemma}[mythm]{Lemma}
\newtheorem{mycor}[mythm]{Corollary}

\theoremstyle{definition}

\newtheorem{mydef}[mythm]{Definition}

\newtheorem{myrmk}[mythm]{Remark}

\newcommand{\ol}[1]{\overline{#1}}
\newcommand{\ot}[1]{\widetilde{#1}}

\newcommand{\Rar}{\Rightarrow}
\newcommand{\rar}{\rightarrow}

\newcommand{\gen}[1]{{\langle #1 \rangle}}

\newcommand{\pres}[2]{\langle #1 \ |\  #2\rangle}
\newcommand{\abs}[1]{|#1|}
\newcommand{\norma}[1]{\|#1\|}

\newcommand{\bC}{\mathbb C}
\newcommand{\bR}{\mathbb R}

\newcommand{\bZ}{\mathbb Z}
\newcommand{\bN}{\mathbb N}

\newcommand{\dpar}{\partial}
\newcommand{\id}{\text{id}}

\newcommand{\len}{\mathrm{Len}}
\newcommand{\area}{\mathrm{Area}}
\newcommand{\scl}[1]{\mathrm{scl}(#1)}

\newcommand{\mappapi}{b\text{\nobreakdash-projection}}

\newcommand*{\Linfs}{{\scriptscriptstyle{(\infty)}}} 
\newcommand*{\st}{\colon\,}

\newcommand\extrafootertext[1]{%
	\bgroup
	\renewcommand\thefootnote{\fnsymbol{footnote}}%
	\renewcommand\thempfootnote{\fnsymbol{mpfootnote}}%
	\footnotetext[0]{#1}%
	\egroup
}

\title{Weakly bounded cohomology classes and a counterexample to a conjecture of Gromov}
\author{
Dario Ascari\\
{\small \textit{Mathematical Institute, Andrew Wiles Building,}}\\
{\small \textit{University of Oxford, Oxford OX2 6GG, UK}}\\
{\small e-mail: \texttt{ascari.maths@gmail.com}}\\
\and
Francesco Milizia\\
{\small \textit{Scuola Normale Superiore, Palazzo della Carovana,}}\\
{\small \textit{Piazza dei Cavalieri, 7, 56126 Pisa, IT}}\\
{\small email: \texttt{francesco.milizia@sns.it}}\\
\\
}

\date{}

\begin{document}

\maketitle

\begin{abstract}
We exhibit a group of type F whose second cohomology contains a weakly bounded, but not bounded, class.
As an application, we disprove a long-standing conjecture of Gromov about bounded primitives of differential forms on universal covers of closed manifolds.
\end{abstract}

\section{Introduction}
\extrafootertext{Ascari was funded by the Engineering and Physical Sciences Research Council.}
\extrafootertext{Milizia has been supported by the INdAM -- GNSAGA project CUP E55F22000270001.}
Let $G$ be a discrete group.
We consider the cohomology of $G$ with coefficients in $A = \bZ$ or $A = \bR$.
In particular, to compute $H^*(G;A)$ we take the bar resolution
\[\begin{tikzcd}
    0 \ar[r] & C^0(G;A) \ar[r,"\delta"] & C^1(G;A) \ar[r,"\delta"] & C^2(G;A) \ar[r,"\delta"] & \dots,
\end{tikzcd}\]
where $C^k(G;A)$ denotes, for every $k \ge 0$, the group of arbitrary set maps from $G^k$ to $A$, and the coboundary maps are defined by the formula
\[(\delta\alpha)(g_1, \dots, g_{k+1})\ =\ \alpha(g_2, \dots, g_{k+1}) + \sum_{i=1}^k (-1)^i \alpha(\dots, g_ig_{i+1}, \dots) + (-1)^{k+1}\alpha(g_1, \dots, g_k).\]

A cohomology class in $H^k(G;A)$ is \textbf{bounded} if it has a representative $\alpha:G^k \to A$ which is bounded, i.e., whose image is a bounded subset of $A$.
We say that a $k$-cochain $\alpha$ is \textbf{weakly bounded} if $\alpha(G,g_2,\dots,g_k)$ is a bounded subset of $A$ for every $g_2, \dots, g_k \in G$.
An element of $H^k(G;A)$ is weakly bounded if it is represented by a weakly bounded cocycle.

Of course, bounded cohomology classes are also weakly bounded, and in degrees $0$ and $1$ the two notions coincide (already at the cochain level).
Neumann and Reeves, motivated by applications in the study of the coarse geometry of central extensions (see below for a brief description of this connection), asked in \cite{NR1996,NR1997} whether a weakly bounded $2$-class on a finitely generated group is always bounded.
Essentially the same question was considered by Whyte in \cite[Remark 2.16]{Whyte2010}.

Intimately related questions were posed by Wienhard and Blank, respectively in \cite[Question 8]{Wienhard2012} and \cite[Question 6.3.10]{Blank2015}.
They asked under what conditions on $G$ a certain sequence of natural maps involving the bounded cohomology, the ordinary cohomology and the $\ell^\infty$-cohomology of $G$ (in some degree $k$) is exact; as shown by Frigerio and Sisto \cite[Proposition 11]{FS2020}, this is equivalent to asking under what conditions on $G$ weakly bounded classes in $H^k(G;A)$ are bounded.
We refer to Section \ref{sec:cohomology} for the definitions of bounded and $\ell^\infty$-cohomology and the precise reformulation of the question, which also comes in handy in later sections.

\

The main result of our paper is the following:
\begin{mythm}\label{mainthm}
    There exists a group $G$ with the following properties:
    \begin{enumerate}
    	\item \label{it:wbnb} There is a cohomology class in $H^2(G;\bZ)$ which is weakly bounded but not bounded;
    	\item \label{it:fp} $G$ is of type F (and in particular finitely presented);
    	\item $G$ is CAT(0).
    \end{enumerate}
\end{mythm}
Frigerio and Sisto provide in \cite[Corollary 10]{FS2020} a finitely generated, but not finitely presented, group satisfying property \ref{it:wbnb}.
Our improvement --- in particular, the fact that we have a finitely \emph{presented} example --- allows us to disprove a conjecture that was proposed by Gromov in \cite{Gromov1993}.
See below for a discussion of this and other applications of our result.
Using standard techniques, starting from Theorem \ref{mainthm} we also construct groups satisfying properties \ref{it:wbnb} and \ref{it:fp} that are non-elementary relatively hyperbolic (see Corollary \ref{cor:rel-hyp}).

\begin{myrmk}
	For a quite large and diverse family of groups $G$, weakly bounded cohomology classes in $H^2(G;\bZ)$ \emph{are} bounded.
	As proved by Frigerio and Sisto \cite{FS2020}, this family is closed under direct and free products (and also some amalgamated products), and includes amenable groups, relatively hyperbolic groups with respect to a finite family of amenable subgroups, right-angled Artin groups and fundamental groups of compact orientable $3$-manifolds.
	
	In higher degrees the situation is different: for every $k \ge 3$ it is easy to construct groups $G$ of type F with weakly bounded but not bounded cohomology classes in $H^k(G;\bZ)$.
	Such examples are given, e.g., by $G = \bZ^{k-2} \times \pi_1(\Sigma_2)$, where $\Sigma_2$ is a closed orientable surface of genus $2$ (see \cite[Corollary 3.3]{FS2020}).
\end{myrmk}

\subsection*{Outline of the proof}
Our construction is unrelated to the finitely generated example by Frigerio and Sisto.
In fact, the property of not being finitely presentable plays a fundamental role in their construction, making it unlikely to manufacture a finitely presented --- let alone of type F --- example by a simple adaptation of their method.

Our example has the following presentation:
\begin{equation}\label{eq:presentation}
	G = \pres{a_1, t_1, a_2, t_2, b}{[t_1,a_1] = [t_2,a_2] = b, [t_1,b] = [t_2,b] = 1},
\end{equation}
where $[x,y]=xyx^{-1}y^{-1}$ denotes the commutator of two elements.
This group $G$ is obtained as an amalgamated free product $H*_\gen{b}H$ of two copies of $H=\pres{a,b,t}{[t,a]=b, [t,b]=1}$ along the infinite cyclic subgroup $\gen{b}$.
The stable commutator length of $b$ in $H$ is equal to $0$, and this is the key property that allows to produce a certain class in $H^2(G;\bZ)$ which is not bounded.
In order to show that the same class is weakly bounded, we rely on a result of \cite{Mil2021}, reducing to prove a certain combinatorial isoperimetric inequality in the Cayley graph of $G$.
To prove that $G$ is CAT(0) and of type F, we explicitly build a finite $2$-dimensional simplicial complex, obtained by gluing a finite number of regular Euclidean triangles, and then show that it is locally CAT(0).

\subsection*{Motivation and applications}
We now discuss some corollaries of Theorem \ref{mainthm}, which also gave us the motivation to study the relation between bounded and weakly bounded classes.

\paragraph{Bounded primitives and Gromov's conjecture}
Let $M$ be a smooth Riemannian manifold.
A smooth differential form $\omega \in \Omega^k(M)$ is bounded if $\sup_{p \in M} \norma{\omega(p)} < +\infty$, where $\norma{\omega(p)}$ is the operator norm of $\omega(p):\wedge^kT_p(M) \to \bR$.

Following \cite{Gromov1991}, we say that $\omega$ is \textbf{$\bm{d}$(bounded)} if it is the exterior differential $\omega = d\eta$ of a bounded form $\eta$.
Let $\ot{M}\to M$ be the universal cover, and endow $\ot{M}$ with the pull-back Riemannian metric.
Then $\omega \in \Omega^k(M)$ is said to be \textbf{$\bm{\ot{d}}$(bounded)} if its pull-back $\ot{\omega} \in \Omega^k(\ot{M})$ is $d$(bounded).

Suppose now that $M$ is a closed connected Riemannian manifold, and let $[\omega] \in H_\mathrm{dR}^2(M;\bR) \cong H^2(M;\bR)$ be a cohomology class represented by a smooth differential $2$-form $\omega$.
Gromov conjectured \cite[page 138]{Gromov1993} that the following conditions are equivalent:
\begin{enumerate}[label=(\roman*)]
    \item\label{it:dbound} $\omega$ is $\ot{d}$(bounded);
    \item\label{it:bound} The class $[\omega] \in H_\mathrm{dR}^2(M;\bR) \cong H^2(M;\bR)$ is represented by a bounded singular cocycle.
\end{enumerate}
From the compactness of $M$ it follows that the first condition does not depend on the differential form $\omega$ representing the class, nor on the Riemannian metric of $M$.
The implication \ref{it:bound} $\Rar$ \ref{it:dbound} holds in any degree, and a self-contained proof of this fact is given by Sikorav \cite{Sik01}.
It is shown by Frigerio and Sisto \cite[Corollary 20]{FS2020} that if a manifold $M$ satisfies Gromov's statement for every $[\omega] \in H_\mathrm{dR}^2(M;\bR)$, then every weakly bounded class in $H^2(\pi_1(M);\bZ)$ is bounded.
Therefore, a counterexample to the conjecture is given by any closed Riemannian manifold whose fundamental group has property \ref{it:wbnb} of Theorem \ref{mainthm}.
Since every finitely presented group can be realized, up to isomorphism, as the fundamental group of a closed manifold, Theorem \ref{mainthm} provides counterexamples for Gromov's conjecture.
In Section \ref{sec:cat0}, exploiting the stronger type-$F$ property, we obtain an \emph{aspherical} counterexample:
\begin{mycor}\label{cor:false-conjecture}
    There exists a closed connected aspherical Riemannian manifold $M$ with a differential $2$-form $\omega \in \Omega^2(M)$ such that \ref{it:dbound} holds but \ref{it:bound} does not.
\end{mycor}

Using a result of Belegradek \cite{Belegradek2006}, we can also tweak the statement of Corollary \ref{cor:false-conjecture} by requiring that the fundamental group of $M$ be non-elementary relatively hyperbolic.
In particular, considering the fundamental group $\pi_1(M)$, we get the following:
\begin{mycor}\label{cor:rel-hyp}
    There exists a group $G$ of type $F$ which is non-elementary relatively hyperbolic and has a cohomology class in $H^2(G;\bZ)$ which is weakly bounded but not bounded.
\end{mycor}

We have to mention, at this point, how Gromov himself was quite cautious about the validity of his statement, as he wrote that ``the evidence in favour of the conjecture is rather limited and it would be safe to make some extra assumptions on $\pi_1(M)$...''.

\begin{myrmk}
	Starting from the work of Gromov \cite{Gromov1991}, the notion of $\ot{d}$(bounded) forms and, thus, that of weakly bounded classes, has gained substantial importance in the context of K\"{a}hler geometry.
	Let $M$ be a closed complex manifold.
	Gromov gave in \cite{Gromov1991} the following definition: $M$ is \textbf{K\"{a}hler hyperbolic} if it admits a K\"{a}hler structure whose $2$-form $\omega \in \Omega^2(M)$ is $\ot{d}$(bounded).

	The following question arises: given a K\"{a}hler hyperbolic manifold, is its K\"{a}hler class bounded (in the sense of \ref{it:bound})?
	We are not able to provide a negative answer using the techniques of this paper.
	In fact, if $G$ is the group defined by the presentation (\ref{eq:presentation}), and $[\alpha] \in H^2(G;\bZ)$ is the weakly bounded --- but not bounded --- class, then $[\alpha] \cup [\alpha] = 0$, because there is a model of $K(G,1)$ of dimension $2$ (we manufacture one explicitly in Section \ref{sec:cat0}). On the other hand, if $M$ is a K\"{a}hler hyperbolic manifold of complex dimension $n \ge 2$ and K\"{a}hler form $\omega$, then $[\omega]^n \neq 0 \in H^{2n}(M;\bR)$, because $\frac{\omega^n}{n!}$ is the volume form of $M$. Therefore, $[\omega]$ cannot coincide with $[\alpha]$.

	One might try instead to pull-back the class $[\alpha]$ via a retraction, as we now explain.
	Let $G'$ be a group having $G$ as a subgroup and admitting a retraction $r:G' \to G$: we observe in Lemma \ref{lemma:retract} that $r^*[\alpha] \in H^2(G';\bZ)$ is weakly bounded and not bounded. However, since $[\alpha]\cup[\alpha]=0$ we also have that $r^*[\alpha] \cup r^*[\alpha] = 0$. Therefore, $[\omega]$ cannot be the pull-back of $[\alpha]$ via a retraction $r:\pi_1(M) \to G$, and this attempt doesn't work either.
	
\end{myrmk}

\paragraph{Quasi-isometrically trivial central extensions}
Let $G$ be a finitely generated group, and let
\[\begin{tikzcd}
	1 \ar[r] & \bZ \ar[r] & E \ar[r] & G \ar[r] & 1
\end{tikzcd}\]
be a central extension of $G$ by $\bZ$.
Associated to such an extension there is an element of $H^2(G;\bZ)$, called the Euler class of the extension.
The Euler class vanishes if and only if the extension is trivial, i.e., there is a commutative diagram of the form
\[\begin{tikzcd}
	1 \ar[r] & \bZ \ar[r]\ar[d,"\id"] & E \ar[r]\ar[d,"h"] & G \ar[r]\ar[d,"\id"] & 1\\
	1 \ar[r] & \bZ \ar[r,"i_1"] & \bZ \times G \ar[r,"\pi_2"] & G \ar[r] & 1
\end{tikzcd}\]
where $h$ is a group isomorphism.
The extension is said to be quasi-isometrically trivial if there is a diagram as above, with the following differences: $h$ is only required to be a quasi-isometry and the squares have to commute only up to a bounded error.
Gersten proved in \cite[\S 3]{Gersten1992} that if the Euler class of the extension is bounded, then the extension is quasi-isometrically trivial.
Then, Neumann and Reeves observed in \cite{NR1996} that the extension is quasi-isometrically trivial if and only if the Euler class is weakly bounded.
A detailed proof of this fact is provided by Frigerio and Sisto \cite{FS2020}.
Thus, we get the following corollary of Theorem \ref{mainthm}:
\begin{mycor}
	There exists a group $G$ of type F such that there is a quasi-isometrically trivial central extension of $G$ by $\bZ$ whose Euler class is not bounded.
\end{mycor}
This corollary follows from the fact that every class in $H^2(G;\mathbb{Z})$ --- in particular, the one provided by Theorem \ref{mainthm} --- can be realized as the Euler class of some central extension of $G$ by $\mathbb{Z}$, and from the observation of Neumann and Reeves mentioned above.
In \cite[Theorem 2]{FS2020} the same phenomenon is exhibited by the finitely generated, but not finitely presented, group constructed by Frigerio and Sisto.

\paragraph{Refinements of bounded cohomology}
In the recent preprint \cite{GK2022}, Gal and Kedra introduce a refinement of bounded cohomology that, in particular, provides an ``interpolation'' between the notions of bounded and weakly bounded classes.

Let $G$ be a group.
For every $p \in \bN$, they define the $p$-bounded cohomology of $G$, denoted by $H_{(p)}^*(G;\bR)$; for every degree $n$ there are natural ``comparison maps'' induced by the inclusion, for any $p$, of ``$p$-bounded cochains'' into ``$(p-1)$-bounded cochains'':
\[ H_{(n)}^n(G;\bR) \to \dots \to H_{(1)}^n(G;\bR) \to H_{(0)}^n(G;\bR).\]
The rightmost (resp.\ leftmost) term is isomorphic to the ordinary (resp.\ bounded) cohomology of $G$ in degree $n$.
The image of $H_{(1)}^n(G;\bR) \to H_{(0)}^n(G;\bR)$ consists precisely of the weakly bounded classes, while the image of $H_{(n)}^n(G;\bR) \to H_{(0)}^n(G;\bR)$ is given by the bounded classes.
Consider the case $n = 2$: then property \ref{it:wbnb} of Theorem \ref{mainthm} implies that these two images can be different.

\paragraph{Structure of the paper}
In Section \ref{sec:cohomology} we recall the definitions of bounded and $\ell^\infty$-cohomology of groups and spaces, and restate property \ref{it:wbnb} of Theorem \ref{mainthm} in terms of some natural maps between them.
In Section \ref{sec:example}, which is the heart of the paper, we show that the group $G$ defined by the presentation (\ref{eq:presentation}) satisfies property \ref{it:wbnb} of Theorem \ref{mainthm}.
Finally, in Section \ref{sec:cat0} we prove that $G$ is CAT(0) and of type F, and obtain Corollaries \ref{cor:false-conjecture} and \ref{cor:rel-hyp}.

\paragraph{Acknowledgements}
We thank Roberto Frigerio and Alessandro Sisto for helpful conversations about this problem and for the idea of considering amalgamated free products.
We are also grateful to Roberto Frigerio (again) for his useful comments on preliminary drafts of the paper, to Jonathan Bowden for pointing out some material about K\"{a}hler hyperbolicity, and to the anonymous referees for many comments that helped us to improve the exposition.

\section{Bounded cohomology and \texorpdfstring{$\bm{\ell^\infty}$}{L-infinity}-cohomology}\label{sec:cohomology}

The purpose of this section is to recall the definitions of bounded cohomology and $\ell^\infty$-cohomology, and to use them to restate property \ref{it:wbnb} of Theorem \ref{mainthm} in a way that will be useful in subsequent sections.

In the definitions below we allow coefficients $A = \bZ$ or $A = \bR$ (with trivial action of the group).
After the definitions for groups, we also consider bounded and $\ell^\infty$-cohomology of spaces, since in Section \ref{sec:example} we prefer to mostly work with spaces (more precisely, CW complexes), instead of groups.

\paragraph{Bounded cohomology of groups}
The study of bounded cohomology is a very active research area which gained popularity after the fundamental paper of Gromov \cite{Gromov1982}.

The bounded cohomology of a group $G$, denoted by $H_b^*(G;A)$, is defined as the cohomology of the subcomplex $C^*_b(G;A) \subseteq C^*(G;A)$ given by bounded cochains.
Here, cochains are functions (not necessarily homomorphisms) from $G^n$ to $A$, for some $n\in\bN$; bounded cochains are functions whose image is a bounded subset of $A$.
Notice that the coboundary operator sends bounded cochains to bounded cochains.
The inclusion at the level of cochains induces the so-called \emph{comparison map} $c^*:H_b^*(G;A) \to H^*(G;A)$ in cohomology.

\paragraph{\bm{$\ell^\infty$}-cohomology of groups}
Gersten introduced $\ell^\infty$-cohomology in \cite{Gersten1992}, using it as a tool to obtain lower bounds for the Dehn function of finitely presented groups.

The $\ell^\infty$-cohomology of a group $G$, denoted by $H_\Linfs^*(G;A)$, is defined as the cohomology of $G$ with coefficients in $\ell^\infty(G,A)$, the $A[G]$-module of bounded functions from $G$ to $A$.
The group $G$ acts on $\ell^\infty(G,A)$ as follows: $(g\cdot f)(h) = f(g^{-1}h)$ for every $f \in \ell^\infty(G,A)$ and every $g,h \in G$.
Recall that, for every $k \in \bN$, $k$-cochains are functions $\alpha:G^k \to \ell^\infty(G,A)$, and that the coboundary operator is given by the following formula, which is slightly more complicated than the one appearing in the introduction, because here $G$ has a non-trivial action on the coefficients:
\[(\delta\alpha)(g_1, \dots, g_{k+1})\ =\ g_1 \cdot (\alpha(g_2, \dots, g_{k+1})) + \sum_{i=1}^k (-1)^i \alpha(\dots, g_ig_{i+1}, \dots) + (-1)^{k+1}\alpha(g_1, \dots, g_k).\]
The inclusion of $A$ into $\ell^\infty(G,A)$ as the subgroup of constant functions induces the change of coefficients map $\iota^*:H^*(G;A) \to H_\Linfs^*(G;A)$ in cohomology.

\paragraph{Definitions for CW complexes}
Let $X$ be a connected CW complex (this assumption is more restrictive than what is needed in the general theory, but it suffices for our purposes).
Denote by $G$ the fundamental group $\pi_1(X,*)$.
\begin{itemize}
    \item Let $C^*(X;A)$ be the complex of singular cochains of $X$; we think of a singular cochain as a function, with values in $A$, defined on the set of singular simplices of $X$ of a certain dimension.
    A cochain is bounded if the set of values it assigns to singular simplices is a bounded subset of $A$.
    This condition defines a subcomplex $C_b^*(X;A) \subseteq C^*(X;A)$; the bounded cohomology of $X$, denoted by $H_b^*(X;A)$, is the cohomology of this subcomplex.
    As in the group-theoretical setting, the inclusion $C_b^*(X;A) \subseteq C^*(X;A)$ induces the comparison map $c^*:H_b^*(X;A) \to H^*(X;A)$.
    \item The $\ell^\infty$-cohomology of $X$, denoted by $H_\Linfs^*(X;A)$, is defined as the equivariant cohomology of the universal cover of $X$ with coefficients in the $A[G]$-module $\ell^\infty(G,A)$.
    We point the reader to, e.g., \cite[Section VI.3]{Whitehead1978} for the definition of equivariant cohomology, and to \cite{Mil2021} for an alternative description of $\ell^\infty$-cohomology as the cohomology of cochains on the universal cover which are ``bounded on orbits''.
    Singular and cellular (equivariant) cohomologies can be used interchangeably, since the two are canonically isomorphic.
    Like before, the inclusion of $A$ into $\ell^\infty(G,A)$ induces the map $\iota^*:H^*(X;A) \to H_\Linfs^*(X;A)$ in cohomology.
\end{itemize}
\begin{myrmk}
    The bounded cohomology of a CW complex \emph{cannot} be computed using cellular cochains.
    This can be illustrated with the following example: the bounded cohomology of the wedge of two circles is nontrivial in degree $2$, despite the absence of $2$-dimensional cells. 
\end{myrmk}

Let $X$ be a model for $K(G,1)$, i.e., a connected CW complex with an isomorphism $\pi_1(X,*) \cong G$ and whose universal cover is contractible.
The general theory of (equivariant) cohomology then implies that there are canonical isomorphisms $H_\Linfs^*(G;A) \cong H_\Linfs^*(X;A)$ and $H^*(G;A) \cong H^*(X;A)$.
There is also a canonical isomorphism $H_b^*(G;A) \cong H_b^*(X;A)$ (if $A=\bR$ even more is true: the bounded cohomology groups are naturally endowed with a seminorm, and the isomorphism preserves it; moreover, there is an isomorphism even without the assumption that the universal cover is contractible. This is the \emph{mapping theorem}, see, e.g., \cite{Ivanov2017}).

These isomorphisms are part of a commutative diagram
\[
\begin{tikzcd}
    H_b^*(G;A) \ar[r,"c^*"]\ar[d] & H^*(G;A) \ar[r,"\iota^*"]\ar[d] & H_\Linfs^*(G;A) \ar[d] \\
    H_b^*(X;A) \ar[r,"c^*"] & H^*(X;A) \ar[r,"\iota^*"] & H_\Linfs^*(X;A)
\end{tikzcd}
\]
where the vertical maps are the canonical isomorphisms.
The square on the right commutes because of the naturality of the change of coefficients $A \subseteq \ell^\infty(G,A)$, which induces the horizontal maps.
For the commutativity of the square on the left see, e.g., \cite{Ivanov2017}.

\paragraph{Reformulation of property \ref{it:wbnb} of Theorem \ref{mainthm}}
Let $G$ be a group.
The existence of a weakly bounded and not bounded class in $H^2(G;\bZ)$ is equivalent to the following:
the sequence of maps
\begin{equation}\label{sequence}
    \begin{tikzcd}
        H_b^2(G;\bZ) \ar[r,"c^2"] & H^2(G;\bZ) \ar[r,"\iota^2"] & H_\Linfs^2(G;\bZ)
    \end{tikzcd}
\end{equation}
is not exact at $H^2(G;\bZ)$.
The equivalence descends immediately from the following facts:
\begin{itemize}
    \item A class in $H^k(G;\bZ)$ is bounded if and only if it lies in the image of $c^k$;
    \item A class in $H^k(G;\bZ)$ is weakly bounded if and only if it lies in the kernel of $\iota^k$.
\end{itemize}
The first of the two facts follows directly from the definitions, while the second one is proved by Frigerio and Sisto \cite[Proposition 11]{FS2020} for coefficients in any finitely generated Abelian group, or the real numbers.
The two facts listed above also hold for real coefficients.

The following result, taken from \cite{FS2020}, will allow us to work with real coefficient even though our goal is to prove that a certain class in $H^2(G;\bZ)$ is weakly bounded and not bounded.
\begin{mylemma}\label{lemma:int_real}
	Let $x \in H^k(G;\bZ)$ and denote by $x_\bR$ the image of $x$ in $H^k(G;\bR)$ under the change of coefficient map.
	Then:
	\begin{itemize}
		\item $x$ is weakly bounded if and only if $x_\bR$ is weakly bounded;
		\item $x$ is bounded if and only if $x_\bR$ is bounded.
	\end{itemize}
\end{mylemma}
\begin{proof}[Sketch of the proof]
	The ``only if'' directions follow directly from the definitions: if $x$ has a (weakly) bounded integer representative, then the same representative witnesses the fact that also $x_\bR$ is (weakly) bounded.
	
	If $x = [\alpha]$ and $x_\bR = [\beta]$ with $\beta$ a (weakly) bounded cochain, then $\beta = \alpha + \delta\eta$ for a certain $\eta \in C^k(G;\bR)$.
	Now, replace $\eta$ with a $\eta' \in C^k(G;\bZ)$ such that $\eta-\eta'$ is bounded.
	The resulting $\beta' = \alpha + \delta\eta'$ represents $x$ and is (weakly) bounded.
\end{proof}

\section{A weakly bounded but not bounded class}\label{sec:example}

In this section we construct a group $G$ together with a cohomology class $[\alpha]\in H^2(G;\bZ)$ which is weakly bounded but not bounded.

The group $G$ is constructed as an amalgamated free product $H*_{\gen{b}}H$ for a certain group $H$ and a certain element $b\in H$, as we will explain in Section \ref{sec:construction}. In Section \ref{sec:class} we provide an explicit model for $K(G,1)$, which we call $X$, and we use it to define the cohomology class $[\alpha]$, which is strictly related to the decomposition of $G$ as amalgamated free product. In Section \ref{sec:unbounded} we prove that $[\alpha_\bR]$ is not bounded: this is done by embedding a genus-$2$ closed surface $S_2$ in $X$ in several different ways, and by taking the cap product between the class $[\alpha_\bR]$ and (the image of) the fundamental class $[S_2]\in H_2(S_2;\bR)$. Finally, Sections \ref{sec:area} and \ref{sec:wb} are dedicated to the proof that $[\alpha_\bR]$ is weakly bounded. We rely on a result taken from \cite{Mil2021}, that characterizes weakly bounded cohomology classes in $H^2(G;\bR)$ in terms of a linear isoperimetric inequality in the universal cover $\ot X$ of $X$: we study the structure of $\ot X$ and we prove that such isoperimetric inequality holds for the class $[\alpha_\bR]$.

\subsection{Construction of the group}\label{sec:construction}

Consider the group $H=\pres{a,b,t}{tbt^{-1}=b,tat^{-1}=ba}$. This is a free-by-cyclic group, and in particular it's an HNN extension of the free group $F_2=\gen{a,b}$ with stable letter $t$. An equivalent way of writing the same presentation is $H=\pres{a,b,t}{[t,b]=1,[t,a]=b}$.

\begin{mylemma}\label{scl0}
For every $n\in\bN$, the element $b^n\in H$ is a commutator.

\end{mylemma}
\begin{proof}
We show by induction on $n$ that $b^n=[t^n,a]$. The base step is trivial. For the inductive step, assume that $b^n=[t^n,a]$ and observe that
$$[t^{n+1},a]\ =\ t[t^n,a]at^{-1}a^{-1}\ =\ tb^nat^{-1}a^{-1}\ =\ b^ntat^{-1}a^{-1}\ =\ b^{n+1}.$$
The conclusion follows.
\end{proof}

We consider the Cayley graph of $H$ with respect to the generating set $\{a,b,t\}$. This graph has a vertex for each element of $H$, and two vertices $h,h'\in H$ are connected by an edge if and only if $h'=hs$ for some $s\in\{a,b,t,a^{-1},b^{-1},t^{-1}\}$. We put on the Cayley graph the path metric that assigns length $1$ to each edge, and we consider on $H$ the metric as a subspace of its Cayley graph.

Since $H$ is free-by-cyclic, every element of $H$ can be uniquely written as $t^pv(a,b)$ where $p\in\bZ$ and $v(a,b)$ is an element of the free group $F_2=\gen{a,b}$. In the free group $F_2=\gen{a,b}$ we can uniquely write $v(a,b)=b^qw(a,b)$ where $q\in\bZ$ and $w(a,b)$ is a reduced word in the letters $a,b$ (and their inverses) such that $w(a,b)$ doesn't begin with $b$ or $b^{-1}$. It follows that every element of $H$ can be uniquely written as $t^pb^qw(a,b)$ where $p,q\in\bZ$ and $w(a,b)$ is a reduced word in the letters $a,b$ (and in their inverses) that doesn't begin with $b$ or $b^{-1}$.

We consider the map $\pi:H\rar\bZ$ given by $\pi(t^pb^qw(a,b))=q$, and call this map \textbf{\bm{$\mappapi$}}.
The $\mappapi$ map $\pi$ has the following two properties, which will be of fundamental importance in what follows: it is $1$-Lipschitz, and its restrictions to $\gen{b}$-cosets are either constant maps or translations.
We prove these properties in two separate lemmas.

\begin{mylemma}\label{pi-lipschitz}
The map $\pi:H\rar\bZ$ is $1$-Lipschitz.
\end{mylemma}
\begin{proof}
Consider an arbitrary element $h=t^pb^qw(a,b)$ where $w(a,b)$ does not begin with $b^{\pm1}$.
We proceed to show that $\abs{\pi(h)-\pi(hs)} \le 1$ for every $s \in \{a,b,t\}$.
There is no need to consider multiplication by $s^{-1}$ separately, since $\abs{\pi(h)-\pi(hs^{-1})} = \abs{\pi(h^\prime)-\pi(h^\prime s)}$ where $h^\prime = hs^{-1}$.

Write the word $t^pb^qw(a,b)a$ and reduce it: we have that at most one cancellation occurs, and, if it is the case, it involves the last two $a$ and $a^{-1}$ letters; in particular the exponent $q$ remains unchanged and thus $\pi(ha)=\pi(h)$.

Similarly, write the word $t^pb^qw(a,b)b$ and reduce it: again, at most one cancellation occurs, and, if it is the case, it involves the last two $b$ and $b^{-1}$ letters; if $w(a,b)$ is non-empty, the exponent $q$ remains unchanged; if $w(a,b)$ is empty, the exponent $q$ changes by exactly one. In any case we have $\abs{\pi(hb)-\pi(h)}\le1$.

Finally, consider the element $t^pb^qw(a,b)t$ of $H$ and notice that it is equal to $t^{p+1}b^qw(b^{-1}a,b)$.
The word $w(b^{-1}a,b)$ is obtained from $w(a,b)$ by substituting each occurrence of $a$ with $b^{-1}a$, and each occurrence of $a^{-1}$ with $a^{-1}b$.
After performing this substitutions letter by letter, we obtain a possibly unreduced word which we can then reduce to $w'(a,b)$.
We observe that, during the reduction process, no $a$ or $a^{-1}$ letter gets cancelled.
If $w(a,b)$ begins with $a$, then $w'(a,b)$ begins with $b^{-1}a$ and thus $\pi(ht)=q-1$; if $w(a,b)$ begins with $a^{-1}$, then $w'(a,b)$ begins with $a^{-1}$ too and thus $\pi(ht)=q$; if $w(a,b)$ is empty, then $\pi(ht) = q$.
In any case we have $\abs{\pi(ht)-\pi(h)}\le1$.
\end{proof}

\begin{mylemma}\label{pi-cosets}
For every element $u\in H$, the restriction $\pi:u\gen{b}\rar\bZ$ to the coset $u\gen{b}$ is of one of the following forms:
\begin{enumerate}[label=(\roman*)]
\item\label{parallel} It is a translation, i.e., $\pi(ub^n) = \pi(u) + n$ for all $n\in\bZ$;
\item\label{non-parallel} It is constant on $u\gen{b}$, i.e., $\pi(ub^n) = \pi(u)$ for all $n\in\bZ$.
\end{enumerate}
Moreover, the coset $\gen{b}$ falls into case \ref{parallel}.
\end{mylemma}
\begin{proof}
Let $u=t^pb^qw(a,b)$ so that $w(a,b)$ is reduced and does not begin with $b$ or $b^{-1}$. By definition $\pi(u)=q$.
If $w(a,b)$ is non-empty, then $\pi(ub^n)=q$ regardless of $n$, and thus we fall into case \ref{non-parallel}.
If $w(a,b)$ is empty, then $\pi(ub^n)=n+q$ and thus we fall into case \ref{parallel}.
Notice that with $u=1$ we fall into case \ref{parallel}.
\end{proof}

The $\mappapi$ map $\pi$ can be thought as a projection from the group $H$ to $\gen{b}\cong\bZ$ (hence the name).
Similarly, we can project $H$ on a generic coset $u\gen{b}$: for $u\in H$, define the map $\pi_u:H\rar\bZ$ given by $\pi_u(h)=\pi(u^{-1}h)$.
From Lemmas \ref{pi-lipschitz} and \ref{pi-cosets} it follows that $\pi_u:H\rar\bZ$ is $1$-Lipschitz and that, for each $u'\in H$, the restriction $\pi_u:u'\gen{b}\rar\bZ$ of $\pi_u$ to the coset $u'\gen{b}$ falls in one of cases \ref{parallel} or \ref{non-parallel} of Lemma \ref{pi-cosets}.

\begin{mylemma}\label{pi-cosets-equivalence}
For two elements $u,u'\in H$, the following are equivalent:
\begin{itemize}
\item The restriction $\pi_u:u'\gen{b}\rar\bZ$ falls into case \ref{parallel} of Lemma \ref{pi-cosets};
\item The restriction $\pi_{u'}:u\gen{b}\rar\bZ$ falls into case \ref{parallel} of Lemma \ref{pi-cosets}.
\end{itemize}
\end{mylemma}
\begin{proof}
    Each of the conditions holds if and only if $u'=ut^pb^q$ for some $p,q\in\bZ$, since $b$ and $t$ commute.
\end{proof}
\begin{mydef}
We say that two cosets $u\gen{b}$ and $u'\gen{b}$ are \textbf{parallel} if they satisfy any (and thus both) of the conditions of Lemma \ref{pi-cosets-equivalence}.
\end{mydef}

We are now ready to provide an example of a group together with a cohomology class which is weakly bounded but not bounded.
The group we consider is $G=H*_\gen{b}H$, the amalgamated free product of two copies of $H$, where we identify the two copies of the subgroup $\gen{b}$.

\begin{myrmk}\label{rmk:hypothesis}
The rest of Section \ref{sec:example} is dedicated to the construction of a cohomology class in $H^2(H*_{\gen{b}}H;\bZ)$ and to the proof that this class is weakly bounded but not bounded.
However, we point out that the exact same construction can be performed (and provides a weakly bounded but not bounded cohomology class) for any group $H$ that satisfies the following properties:
\begin{itemize}
\item We require that we are given a group $H$ and an element $b\in H$ of infinite order;
\item For the element $b$, we look at its \textbf{stable commutator length} in $H$: for $n\in\bN$, let $r_n\ge0$ be the minimum non-negative integer such that $b^n$ can be written as a product of $r_n$ commutators in $H$, and define $$\scl{b}=\liminf_{n\rar+\infty}\frac{r_n}{n}.$$
We require that $\scl{b}=0$;
\item We require that we are given a \textbf{\bm{$\mappapi$}} map $\pi:H\rar\bZ$ such that:
\begin{itemize}
	\item $\pi$ is Lipschitz with respect to the word metric induced by some generating set of $H$;
	\item For every element $u \in H$, the restriction $\pi:u\gen{b}\to\bZ$ is either constant or a translation (see Lemma \ref{pi-cosets}), possibly up to bounded error.
	More precisely, there is a constant $C$ (independent of $u$) such that, for every $u \in H$, the restriction $\pi:u\gen{b}\to\bZ$ either has image of diameter $\le C$, or satisfies $\abs{(\pi(ub^n)-\pi(ub^m)) - (n - m)} \le C$.
	For $u = 1$ the restriction must be of the translation type.
\end{itemize}
\end{itemize}
For the group $H$ constructed in this section, the property $\scl{b}=0$ follows from Lemma \ref{scl0}, the $b$-projection map has Lipschitz constant $1$ (Lemma \ref{pi-lipschitz}), and its restrictions are honest constants or translations, with error $C = 0$ (Lemma \ref{pi-cosets}).

Our group $H$ and $b$-projection $\pi$ also satisfy the conclusion of Lemma \ref{pi-cosets-equivalence}, which allows to define the relation of parallelism among $\gen{b}$-cosets.
This property isn't strictly necessary for the construction, but makes some of the arguments in Subsection \ref{sec:wb} easier, and we won't hesitate in using it.
In general one would still have a relation of parallelism, but it may fail to be symmetric.
\end{myrmk}

\subsection{Construction of the cohomology class}\label{sec:class}

In this subsection, we produce a specific model for $K(G,1)$, i.e., an aspherical space with fundamental group $G$, by gluing two copies of a model for $K(H,1)$ with a cylinder, and we use it to define a certain cohomology class in $H^2(G;\bZ)$.
As we will see later, this cohomology class will turn out to be weakly bounded but not bounded.

In order to build a model for $K(H,1)$, we perform the following standard construction. Let $Y$ be a CW complex consisting of one $0$-cell, three $1$-cells, oriented and labeled with $a,b,t$, two $2$-cells glued along the paths $tbt^{-1}b^{-1}$ and $tat^{-1}a^{-1}b^{-1}$, and cells in higher dimension in such a way that all the homotopy groups $\pi_k(Y,*)$ for $k\ge2$ are trivial. From the construction, it is obvious that $Y$ is a model for $K(H,1)$.

We denote by $\ot Y$ the universal cover of $Y$. Each edge in the $1$-skeleton of $\ot Y$ inherits a label and an orientation (based on which edge of $Y$ it is mapped to by the covering map). Observe that the $1$-skeleton of $\ot Y$ is exactly the Cayley graph of $H$ with respect to the generating set $a,b,t$ (see Figure \ref{fig:ytilde}). If we fix a basepoint in the $0$-skeleton of $\ot Y$, this allows us to identify the $0$-cells of $\ot Y$ with the elements of $H$: the basepoint corresponds to the identity element of $H$, and crossing a $1$-cell corresponds to right multiplication by the label of the edge or its inverse, according to the orientation of the $1$-cell.

Consider the subspace of the $1$-skeleton of $\ot Y$ given by the union of all the (closures of the) edges labeled $b$: each connected component of such subspace is called a \textbf{$\bm{b}$-line}, and corresponds to a coset $u\gen{b}$ for some $u\in H$.
We say that two $b$-lines are \textbf{parallel} if the two corresponding cosets are parallel (this does not depend on the choice of the basepoint).

\begin{figure}[H]
    \centering
    \includegraphics[scale=0.6]{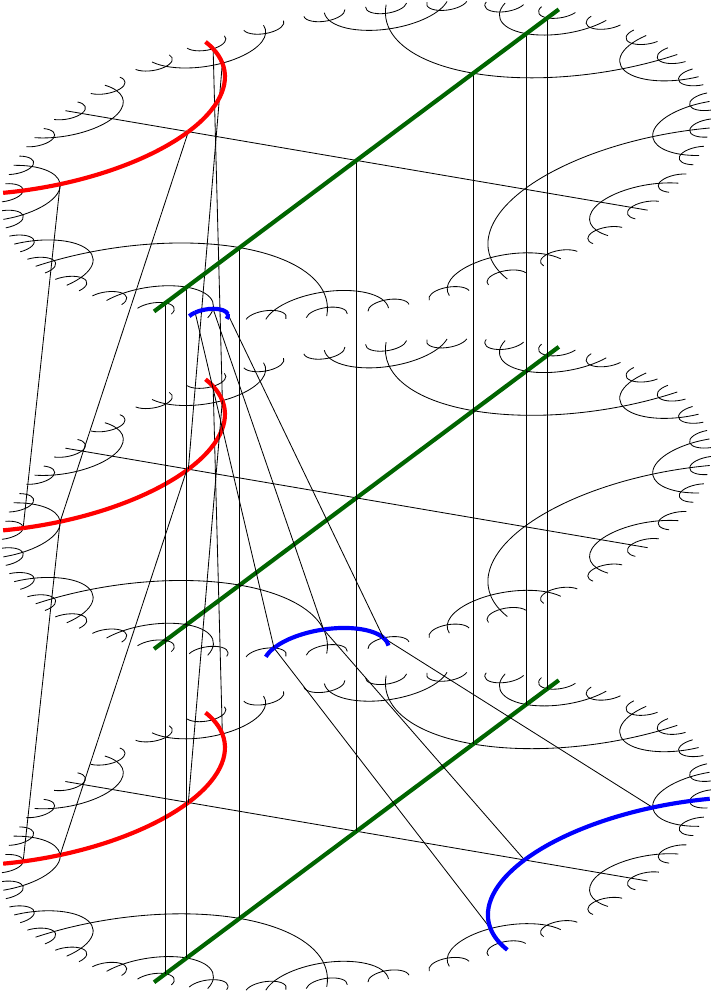}
    \caption{A portion of the $1$-skeleton of $\ot Y$, which coincides with the Cayley graph of $H$; the labeling has been omitted.
        The horizontal leaves are copies of the Cayley graph of $F_2 = \gen{a,b}$, and are joined by $1$-cells labeled with $t$ and oriented upward.
        For simplicity, only some of these $1$-cells are drawn.
        In bold we can see some $b$-lines, parallel $b$-lines having the same color.}
    \label{fig:ytilde}
\end{figure}

Consider now the group $G=H*_\gen{b}H$.
Let $S^1=\{z\in\bC \st \abs{z}=1\}$ be the unit circle, and consider on $S^1\times[0,1]$ a structure of CW complex with two $0$-cells $(1,0),(1,1)$, three $1$-cells $S^1\times\{0\},S^1\times\{1\},\{1\}\times[0,1]$ and a single $2$-cell.
Let $X = Y\cup_bY$ be the CW complex obtained by taking two copies $Y_0,Y_1$ of $Y$ and a copy of $S^1\times[0,1]$, by gluing $S^1\times\{i\}$ to $Y_i$ along the $1$-cell labeled $b$, for $i \in \{0,1\}$ (see Figure \ref{fig:X}). 
We observe that $\pi_1(X)\cong G$, and it follows from the classical result \cite[Corollary on page 160]{Whitehead1939} that $X$ is aspherical, and thus a model for $K(G,1)$, since it is obtained from aspherical spaces glued along aspherical and $\pi_1$-injective subspaces.

\begin{figure}[ht]
    \centering
    \includegraphics[scale=1]{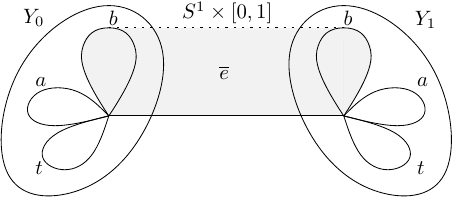}
    \caption{The space $X$ consists of two copies $Y_0,Y_1$ of $Y$, together with a cylinder $S^1\times[0,1]$. For simplicity, in the figure we only drew the $1$-skeleton of $Y_0$ and $Y_1$. The boundary component $S^1\times\{0\}$ (resp.\ $S^1\times\{1\}$) of the cylinder is glued onto the $1$-cell labeled $b$ in $Y_0$ (resp.\ $Y_1$).}
    \label{fig:X}
\end{figure}

We denote by $\ot X$ the universal cover of $X$; it consists of infinitely many disjoint copies of $\ot{Y}$ and \textbf{strips}, i.e., subspaces homeomorphic to $\bR \times [0,1]$ covering the cylinder $S^1\times[0,1]$.
Each side of each strip $\bR\times[0,1]$ is glued along some $b$-line contained in some copy of $\ot Y$, and each $b$-line has exactly one side of one strip glued onto it (see Figure \ref{fig:Xtilde}).
The copies of $\ot Y$ and the strips $\bR\times[0,1]$ are glued in a ``tree-like'' fashion, as we now explain.
Take the space $\ot X$ and collapse each copy of $\ot Y$ to a single point; also, collapse each strip $\bR\times[0,1]$ to a segment by taking the projection on the second component; we obtain a quotient space $T$ which is a graph, with one vertex corresponding to each copy of $\ot Y$, and one edge corresponding to each strip.
The graph $T$ is a tree, since $\ot X$ is simply connected, and each vertex has valence $\aleph_0$.
It is the Bass-Serre tree corresponding to the amalgamated product $H*_\gen{b} H$.
Call $\tau:\ot X\rar T$ the quotient map.

\begin{figure}[ht]
    \centering
    \includegraphics[scale=1]{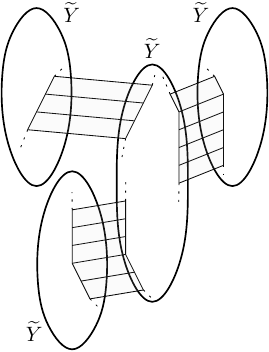}
    \caption{A portion of the space $\ot X$, with some disjoint copies of $\ot Y$ and some strips isomorphic to $\bR\times[0,1]$. Each side of each strip is glued onto a $b$-line in some copy of $\ot Y$.}
    \label{fig:Xtilde}
\end{figure}
%

Let $\ol{e}$ be the $2$-cell of $X$ coming from the unique $2$-cell of $S^1\times[0,1]$. We consider the cellular cohomology of the complex $X$; let $\alpha\in C^2_\mathrm{CW}(X;\bZ)$ be the map given by $\alpha(\ol{e})=1$ and $\alpha(e)=0$ for every other $2$-cell $e\not=\ol{e}$.
We observe that, since no $3$-cell of $X$ is attached on $\ol{e}$, we have $\delta\alpha=0$ and thus $\alpha$ defines a cohomology class $[\alpha]\in H^2_\mathrm{CW}(X;\bZ)$.
Since $X$ is a model for $K(G,1)$, we have a canonical isomorphism between $H^2_\mathrm{CW}(X;\bZ)$ and $H^2(G;\bZ)$, and thus we obtain a class in $H^2(G;\bZ)$; our goal is to show that this class is weakly bounded but not bounded.

\begin{myrmk}
	Associated to the amalgamated free product $G = H *_\gen{b} H$ there is a Mayer-Vietoris sequence, in which we find a homomorphism $H^1(\gen{b};\bZ) \to H^2(G;\bZ)$.
	The class we have produced in $H^2(G;\bZ)$ coincides with the image of the positive generator of $H^1(\gen{b};\bZ) \cong \bZ$ under this homomorphism.
\end{myrmk}



\paragraph{From integral to real coefficients.}
We define $\alpha_\bR \in C_\mathrm{CW}^2(X;\bR)$ to be the cochain corresponding to $\alpha$ under the change of coefficients map induced by the inclusion $\bZ \subseteq \bR$.
By Lemma \ref{lemma:int_real}, passing to real coefficients does not interfere with boundedness or weak boundedness of cohomology classes.
Therefore, to establish property \ref{it:wbnb} of Theorem \ref{mainthm}, it is enough to show that $[\alpha_\bR]\in H^2_\mathrm{CW}(X;\bR) \cong H^2(G;\bR)$ is weakly bounded but not bounded.
Hereafter, when we say that a class in cellular cohomology is bounded or weakly bounded, we mean that the mentioned property is enjoyed by the corresponding class in the cohomology of $G$ via the canonical isomorphism.

\subsection{The cohomology class is not bounded}\label{sec:unbounded}

We now prove that the cohomology class that we have constructed is not bounded.

\begin{myprop}\label{notbounded}
The cohomology class $[\alpha_\bR] \in H^2_\mathrm{CW}(X;\bR) \cong H^2(G;\bR)$ is not bounded.
\end{myprop}
\begin{proof}
Let $S_2$ denote the closed orientable surface of genus $2$, and let $S_{1,1}$ denote the compact orientable surface of genus $1$ and with one boundary component (i.e., the torus with a hole). Notice that the fundamental group of $S_{1,1}$ is a free group of rank $2$ and the boundary $\dpar S_{1,1}$ corresponds to the conjugacy class of the commutator between two generators; thus, for every path-connected topological space $Z$ and for every loop $\zeta$ in $Z$, there is a continuous map from $S_{1,1}$ to $Z$ sending $\dpar S_{1,1}$ to $\zeta$ if and only if the homotopy class of $\zeta$ is a commutator in the fundamental group $\pi_1(Z)$ (this doesn't depend on the choice of a basepoint in $Z$).

Now fix $n\in\bN$: since by Lemma \ref{scl0} the element $b^n$ is a commutator, there is a continuous map $\varphi:S_{1,1}\rar Y$ that restricts to a degree-$n$ map between $\dpar S_{1,1}$ and the $1$-cell labeled by $b$. Let also $\theta:S^1\times[0,1]\rar S^1\times[0,1]$ be the map given by $(e^{ix},y)\mapsto(e^{inx},y)$ (i.e., the map that wraps the cylinder $n$ times around itself).

We now take the two copies $Y_0,Y_1$ of $Y$, and we consider two copies of the surface $S_{1,1}$ along with the two copies $\varphi_0:S_{1,1}\rar Y_0$ and $\varphi_1:S_{1,1}\rar Y_1$ of the map $\varphi$. We consider the surface $S_{1,1}\cup(S^1\times[0,1])\cup S_{1,1}$ obtained by gluing the two boundary components of the two copies of $S_{1,1}$ with the two boundary components of the cylinder. Notice that when performing the gluing process, we have to choose the orientation of the gluing maps on the boundaries. For $i\in\{0,1\}$, we orient the boundary of the $i$-th copy of $S_{1,1}$ in such a way that $\varphi_i:S_{1,1}\rar Y_i$ sends the boundary $\dpar S_{1,1}$ to $b^n$ (and not to $b^{-n}$). For $i\in\{0,1\}$ we give $S^1\times\{i\}$ the standard orientation of $S^1$. We now glue the boundary of the $i$-th copy of $S_{1,1}$ around $S^1\times\{i\}$ preserving the orientation.

We define the map $\psi:S_{1,1}\cup(S^1\times[0,1])\cup S_{1,1}\rar Y\cup_bY$ that coincides with $\varphi_0:S_{1,1}\rar Y_0$ and with $\varphi_1:S_{1,1}\rar Y_1$ on the two copies of the surface $S_{1,1}$, and that coincides with $\theta:S^1\times[0,1]\rar S^1\times[0,1]$ on the cylinder.
The map $\psi$ is well defined, since the maps $\varphi_0,\theta,\varphi_1$ agree on the subspaces which are glued together. Notice that $S_{1,1}\cup(S^1\times[0,1])\cup S_{1,1}$ is homeomorphic to $S_2$, and thus in particular we obtain a map $\psi:S_2\rar X$. Also notice that, by definition, the map $\psi:S_2\rar X$ covers the cell $\ol{e}$ with degree $n$.

Let $[S_2]_\mathrm{CW}^*\in H^2_\mathrm{CW}(S_2;\bR)$ be the real fundamental coclass in the cellular cohomology of $S_2$.
Since $\psi$ covers the cell $\ol{e}$ with degree $n$, we have $\psi^*[\alpha_\bR] = n [S_2]_\mathrm{CW}^*$.
Suppose by contradiction that $[\alpha_\bR] \in H^2_\mathrm{CW}(X;\bR) \cong H^2(G;\bR)$ is bounded.
This implies that there is a bounded \emph{singular} cocycle $\alpha_s \in C^2(X;\bR)$ whose class $[\alpha_s] \in H^2(X;\bR)$ corresponds to $[\alpha_\bR]$ under the canonical isomorphism between singular and cellular cohomology.
Under this isomorphism, $[S_2]_\mathrm{CW}^*$ corresponds to $[S_2]^* \in H^2(S_2;\bR)$, the real fundamental coclass in singular cohomology.
Since the diagram
\[
\begin{tikzcd}[contains/.style = {draw=none,"\in" description}, belongs/.style = {draw=none,"\ni" description}]
	n[S_2]^*_\mathrm{CW} \ar[r,contains] & H^2_\mathrm{CW}(S_2;\mathbb{R}) \ar[d,"\cong"] & H_\mathrm{CW}^2(X;\mathbb{R})  \ar[l,"\psi^*"]\ar[d,"\cong"] & {[\alpha_\mathbb{R}]} \ar[l,belongs]\\
	n[S_2]^* \ar[r,contains] & H^2(S_2;\mathbb{R}) & H^2(X;\mathbb{R}) \ar[l,"\psi^*"] & {[\alpha_s]} \ar[l,belongs]
\end{tikzcd}
\]
commutes, we also have $\psi^*[\alpha_s] = n [S_2]^*$.

Let $L \in \bR$ be such that $\abs{\alpha_s(\sigma)} \le L$ for every singular $2$-simplex $\sigma$ in $X$, and notice that also the norm of the pull-back $\psi^*\alpha_s \in C^2(S_2;\bR)$ is bounded by the same constant $L$.
Let $[S_2]\in H_2(S_2;\bR)$ be the fundamental class of $S_2$ in singular cohomology, and let $C \in \bR$ be the $\ell^1$-norm of a representative of $[S_2]$.
Since $\psi^*[\alpha_s] = n [S_2]^*$, we have that
\[n\ =\ \langle \psi^*[\alpha_s],\ [S_2] \rangle\ \le\ LC,\]
where $\langle -,- \rangle$ denotes the Kronecker pairing.
Since this cannot hold for all $n\in\bN$, we get a contradiction.
\end{proof}

\begin{myrmk}
In the proof of Proposition \ref{notbounded} we use the fact that $b^n$ is a commutator for every $n\in\bN$, but this hypothesis can be relaxed; in fact, the same result can be obtained using only the weaker hypothesis that the stable commutator length of $b$ is zero (see Remark \ref{rmk:hypothesis} for the definition of stable commutator length).
The proof is essentially the same: if $b^n$ is the product of $r_n$ commutators, then one can use in place of $S_{1,1}$ the oriented surface with genus $r_n$ and one boundary component.
\end{myrmk}

\begin{myrmk}
	Since by Proposition \ref{notbounded} the cohomology class $[\alpha_\bR]$ is not bounded, it follows that the class $[\alpha_\bR]$ (and thus also $[\alpha]$) is non-trivial.
	We point out that, in order to obtain that $[\alpha_\bR]$ is non-trivial, a weaker condition on $b$ is sufficient: it is enough to assume that the stable commutator length of $b$ is finite; in other words, that there is a power of $b$ which is a product of commutators in $H$.
\end{myrmk}

\subsection{The cohomology class and the area function}\label{sec:area}

It remains to prove that the class $[\alpha_\bR] \in H^2_\mathrm{CW}(X;\bR) \cong H^2(G;\bR)$ is weakly bounded.
The key element to the proof will be a result in \cite{Mil2021} that relates weakly bounded classes to a linear isoperimetric inequality in the universal cover $\ot X$.
In this subsection, we introduce a notion of ``area'' of combinatorial circuits in $\ot X$, and we show how it is related to our cohomology class.

We begin by defining the combinatorial objects we are going to work with.
A \textbf{path} in a CW complex is a finite non-empty list of $0$-cells $p = (p_0, p_1, \dots, p_k)$ such that, for every $i \in \{1,\dots,k\}$, we have $p_{i-1}\not=p_i$ and there is a $1$-cell having $p_{i-1}$ and $p_{i}$ as endpoints.
We denote by $\len(p)$ its length, which in the example above is equal to $k$.
We call $p$ a \textbf{circuit} if its first and last $0$-cells coincide.

Recall that $X = Y \cup_b Y$ is obtained by gluing a cylinder $S^1\times [0,1]$ and two copies $Y_0, Y_1$ of $Y$.
We call $l = \{1\} \times [0,1] \subseteq S^1\times [0,1]$ the only $1$-cell not contained in $Y_0$ or $Y_1$.
We also fix on $l$ the usual orientation of $[0,1]$; intuitively, the positive orientation corresponds to going away from $Y_0$ towards $Y_1$.

Recall that $\ot{X}$ consists of infinitely many disjoint copies of $\ot{Y}$ and strips homeomorphic to $\bR \times [0,1]$ covering the cylinder $S^1\times[0,1]$.
The covering map $\ot{X}\to X$ induces a cellular structure on every strip, in the same way as the covering map from $\bR\times[0,1]$ to $S^1\times[0,1]$ induces a cellular structure on $\bR\times[0,1]$.
Consider a strip $s$, and choose a homeomorphism $s\cong\bR\times[0,1]$ preserving the covering map to $S^1\times[0,1]$ (and thus preserving the cellular structure too); there are infinitely many such homeomorphisms: we choose one.
For every $n \in \mathbb{Z}$, define the $1$-cell $l_n$ in $s$ to be the one corresponding to the $1$-cell $\{n\}\times[0,1]\subseteq\bR\times[0,1]$ (see Figure \ref{fig:area}).
Notice that $l_n$ is a lifting of $l$, and in particular we can give $l_n$ the same orientation as $l$, for all $n\in\bZ$.

\begin{mydef}\label{def-area}
    Let $p = (p_0, p_1, \dots, p_k)$ be a circuit in $\ot{X}$.
    Let $s$ be a strip in $\ot{X}$ and choose an homeomorphism $s\cong\bR\times[0,1]$ preserving the covering map to $S^1\times[0,1]$.
    We define $\area_s(p) \in \bZ$ by summing the following contributions, for $i \in \{1, \dots, k\}$:
    \begin{enumerate}
        \item\label{positive} If $\partial l_n = p_{i} - p_{i-1}$ for some $n \in \bZ$, we add $n$;
        \item\label{negative} If $\partial l_n = - p_{i} + p_{i-1}$ for some $n \in \bZ$, we add $-n$;
        \item Otherwise, the index $i$ does not contribute to the sum.
    \end{enumerate}
\end{mydef}

\begin{myrmk}\label{rmk:plus_minus}
	Let $p$ and $s$ be as in Definition \ref{def-area}.
	Recall that there is a projection $\tau:\ot X\rar T$ of $\ot X$ onto a tree, and that $s$ is sent to an edge of $T$. Removing the interior of $\tau(s)$ divides $T$ into two connected components $T_0,T_1$. Whenever $p$ goes from $T_0$ to $T_1$ (i.e., for each $i$ such that $\tau(p_i)\in T_1$ and $\tau(p_{i-1})\in T_0$) we have a summand in the definition of $\area_s(p)$ according to case \ref{positive} of Definition \ref{def-area}; whenever $p$ goes from $T_1$ to $T_0$ we have a summand in the definition of $\area_s(p)$ according to case \ref{negative} of Definition \ref{def-area}. The number of times $p$ goes from $T_0$ to $T_1$ must be the same as the number of times $p$ goes from $T_1$ to $T_0$: this means that in the sum defining $\area_s(p)$ the cases \ref{positive} and \ref{negative} occur the same number of times.
\end{myrmk}

\begin{figure}[ht]
    \centering
    \includegraphics[scale=1]{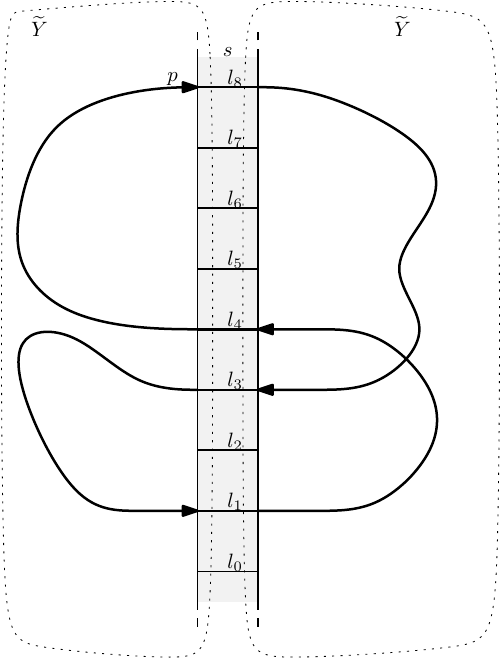}
    \caption{An example of a circuit $p$ that crosses a strip $s$ four times. Edges $l_1,l_8$ are crossed with positive orientation, while edges $l_3,l_4$ are crossed with negative orientation; thus in this case $\area_s(p)=1+8-3-4=2$.
    Informally, $\area_s(p)$ is equal to the number of squares of $s$ ``enclosed'' by $p$, counted with sign.}
    \label{fig:area}
\end{figure}

\begin{mylemma}\label{lemma:area_s-independent}
The value of $\area_s(p)$ does not depend on the chosen homeomorphism $s\cong\bR\times[0,1]$.
\end{mylemma}
\begin{proof}
Suppose we are given two homeomorphisms $\eta_1,\eta_2:s\rar\bR\times[0,1]$, and suppose that each of them commutes with the covering space projections $s\rar S^1\times[0,1]$ and $\bR\times[0,1]\rar S^1\times[0,1]$. Then $\eta_1,\eta_2$ only differ by an integer translation along the $\bR$ component, i.e., $\eta_2=\eta_1+(r,0)$ for some $r\in\bZ$.

If we compute $\area_s(p)$ using the identification $\eta_2$ instead of $\eta_1$, each summand coming from case \ref{positive} increases by $r$, and each summand coming from case \ref{negative} decreases by $r$. Since cases \ref{positive} and \ref{negative} occur the same number of times (Remark \ref{rmk:plus_minus}), the sum $\area_s(p)$ remains unchanged, as desired.
\end{proof}

The above Definition \ref{def-area} gives us a \textit{local} notion of area, related to a given strip. We now define a \textit{global} notion of area, given by the sum of all the local areas.

\begin{mydef}
We define $\area(p) = \sum_s \area_s(p)$ where $s$ varies among all the strips in $\ot{X}$.
\end{mydef}

\begin{myrmk}\label{rmk:finite}
	A circuit $p$ can touch only a finite number of strips, because it consists of a finite list of vertices, and each vertex belongs to exactly one strip.
	In particular the sum defining $\area(p)$ has a finite number of non-zero summands.
\end{myrmk}

If $p$ is a path in $\ot{X}$, it uniquely determines a sequence of $1$-cells; we denote by $\ol{p} \in C_1^\mathrm{CW}(\ot{X};\bR)$ the cellular $1$-chain given by the sum of these cells, with a sign depending on the direction in which the $1$-cell is crossed.
If $p$ is a circuit, then $\ol{p}$ is a $1$-cycle, and since $\ot{X}$ is simply connected, this implies that there is $c \in C_2^\mathrm{CW}(\ot{X};\bR)$ such that $\ol{p} = \partial c$.

We denote by $\ot{\alpha}_\bR \in C_\mathrm{CW}^2(\ot{X};\bR)$ the pull-back of $\alpha_\bR \in C_\mathrm{CW}^2(X;\bR)$ via the covering map.

\begin{myrmk}\label{rmk:alpha-exact}
Since all the homotopy groups of $\ot{X}$ are trivial, by Hurewicz's theorem all the homology groups of $\ot{X}$ are trivial too. Since $\alpha_\bR$ is closed, we have that $\ot{\alpha}_\bR$ is closed too; but since the homology groups of $\ot{X}$ are all trivial, this implies that $\ot{\alpha}_\bR$ is exact.
\end{myrmk}

\begin{mylemma}\label{alpha_area}
    Let $p = (p_0, p_1, \dots, p_k)$ be a circuit in $\ot{X}$.
    Denote by $\ol{p} \in C_1^\mathrm{CW}(\ot{X};\bR)$ the cellular $1$-cycle induced by $p$.
    Let $c \in C_2^\mathrm{CW}(\ot{X};\bR)$ be such that $\ol{p} = \partial c$.
    Then $\ot{\alpha}_\bR(c) = \area(p)$.
\end{mylemma}
\begin{proof}
    Let $s \cong \bR\times [0,1]$ be a strip in $\ot{X}$.
    As in Definition \ref{def-area}, we denote by $l_n$ the $1$-cell $\{n\}\times[0,1] \subseteq s$.
    Let $I^+ \subseteq \{1,\dots,k\}$ be the subset of indices $i$ such that $p_i - p_{i-1} = \partial l_n$ for some $n \in \bZ$, and set $f(i) = n$ for such indices.
    That is, $i \in I^+$ if the $i$-th step of $p$ crosses $s$ positively along the $1$-cell $l_{f(i)}$.
    Similarly, let $I^- \subseteq \{1,\dots,k\}$ be the subset of indices $i$ with $-p_i + p_{i-1} = \partial l_n$ for some $n \in \bZ$, and set $f(i) = n$ for such indices.
    By definition, we have
    \[ \area_s(p)\ =\ \sum_{i \in I^+} f(i) - \sum_{i \in I^-} f(i). \]
    
    For every $n \in \bZ$ we denote by $Q_n = [n,n+1]\times[0,1] \subseteq s$ the $2$-cell between $l_n$ and $l_{n+1}$ (see Figure \ref{fig:area}).
    We orient $Q_n$ in such a way that $l_n$ and $l_{n+1}$ appear respectively with coefficients $-1$ and $+1$ in $\partial Q_n$.
    Consider the following $2$-chain:
    \[c_s\ =\ \sum_{n \in \bZ}\left( \#\{i \in I^+ \st f(i) \ge n+1\} - \#\{i \in I^- \st f(i) \ge n+1\} \right) Q_n.\]
    The sum has a finite number of non-zero terms, because if $n$ is small enough (say, $n < -M$ for a suitable integer $M$), then the two terms whose difference is the coefficient of $Q_n$ are the cardinalities of $I^+$ and $I^-$, which are equal (Remark \ref{rmk:plus_minus}); on the other hand, if $n$ is large enough then both terms vanish.
    We now evaluate $\ot{\alpha}_\bR$ at $c_s$:
    \begin{align*}
        \ot{\alpha}_\bR(c_s)\ &=\ \sum_{n \in \bZ} \left( \#\{i \in I^+ \st f(i) \ge n+1\} - \#\{i \in I^- \st f(i) \ge n+1\} \right)\\
        &=\ \sum_{n \ge -M} \left( \#\{i \in I^+ \st f(i) \ge n+1\} - \#\{i \in I^- \st f(i) \ge n+1\} \right)\\
        &=\ \sum_{i \in I^+} \#\{n \ge -M \st f(i) \ge n+1\} - \sum_{i \in I^-} \#\{n \ge -M \st f(i) \ge n+1\}\\
        &=\ \sum_{i \in I^+} (f(i)+M) - \sum_{i \in I^-} (f(i)+M)\\
        &=\ M (\#I^+ - \#I^-) + \sum_{i \in I^+} f(i) - \sum_{i \in I^-} f(i)\\
        &=\ \area_s(p).
    \end{align*}
    
    The $2$-chain $c_s$ also has the following important property:
    for every $n \in \bZ$, the $1$-cell $l_n$ appears in $\partial c_s$ with coefficient
    $\#\{i \in I^+ \st f(i) = n\} - \#\{i \in I^- \st f(i) = n\}$,
    which is equal to the coefficient of $l_n$ in $\ol{p} = \partial c$.
    
    Let $c' = \sum_s c_s$, with $s$ varying in the set of strips in $\ot{X}$.
    The sum is finite because $c_s = 0$ unless the circuit $p$ crosses $s$ at some point, and $p$ crosses only finitely many strips (Remark \ref{rmk:finite}).
    The $1$-cycle $\partial(c-c')$ is supported in the (disconnected) subspace $\sqcup\ot{Y} \subseteq \ot{X}$ whose components are the various copies of $\ot{Y}$.
    Since every component of $\sqcup\ot{Y}$ is simply connected, there is a $2$-chain $c''$ with support in $\sqcup\ot{Y}$ such that $\partial c'' = \partial(c-c')$.
    This implies that $\ot{\alpha}_\bR(c'') = \ot{\alpha}_\bR(c-c')$, since $\ot{\alpha}_\bR$ is exact by Remark \ref{rmk:alpha-exact}.
    We now have
    \[\ot{\alpha}_\bR(c)\ =\ \ot{\alpha}_\bR(c'') + \ot{\alpha}_\bR(c')\ =\ 0 + \sum_{s\text{ strip}}\ot{\alpha}_\bR(c_s)\ =\ \sum_{s\text{ strip}}\area_s(p)\ =\ \area(p),\]
    which is the desired equality.
\end{proof}

\subsection{The cohomology class is weakly bounded}\label{sec:wb}

In this section we prove a linear isoperimetric inequality for the notion of area introduced in the previous section. We use this inequality to show that $[\alpha_\bR] \in H^2_\mathrm{CW}(X;\bR) \cong H^2(G;\bR)$ is weakly bounded.


Let $p=(p_0,\dots,p_k)$ be a circuit in $\ot X$ and let $V$ be a copy of $\ot Y$ in $\ot X$: we define $\len_V(p)$ to be the number of indices $i\in\{1,\dots,k\}$ such that $p_{i-1},p_{i}$ both belong to $V$.

\begin{mylemma}\label{area-length}
Let $p=(p_0,\dots,p_k)$ be a circuit in $\ot X$. 
Let $V$ be a copy of $\ot Y$ in $\ot X$, and let $\lambda$ be a $b$-line in $V$.
Denote by $S_\lambda$ the set of the strips $s$ such that one side of $s$ is glued on a $b$-line in $V$ which is parallel to $\lambda$.
Then we have
$$\left\lvert \sum\limits_{s\in S_\lambda}\area_s(p) \right\rvert\ \le\ \len_V(p).$$
\end{mylemma}
\begin{proof}
Recall that $X=Y_0\cup_bY_1$ and without loss of generality assume that $V$ projects to $Y_0$. Also assume that $p_0=p_k$ belongs to $V$. Choose a basepoint in $V$ belonging to the $b$-line $\lambda$; this gives an identification between the vertices of $V$ and the elements of $H$, and in particular the $\mappapi$ map $\pi:H\rar\bZ$, as defined in Subsection \ref{sec:construction}, induces a map from the set of the vertices of $V$ to $\bZ$.

Let $J$ be the set of indices $j\in\{0,\dots,k\}$ such that $p_j\in V$. We subdivide $J$ into segments of consecutive indices, i.e., we consider the unique writing $J=[x_0,y_0]\cup\dots\cup[x_r,y_r]$ for some $r,x_0,y_0,\dots,x_r,y_r\in\bN$ with $0=x_0\le y_0\le x_1\le y_1\le\dots\le x_r\le y_r=k$ and with $x_{i}-y_{i-1}\ge 2$ for $i=1,\dots,r$. This means that $p_{x_0},\dots,p_{y_0},p_{x_1},\dots,p_{y_1},\dots,p_{x_r},\dots,p_{y_r}$ is the set of vertices of $p$ that belong to $V$, written with the indices in increasing order. In particular we have $$\sum\limits_{i=0}^r\abs{y_i-x_i}\ =\ \len_V(p).$$

We now look at the sequence of integers $\pi(p_{x_0}),\pi(p_{y_0}),\pi(p_{x_1}),\pi(p_{y_1}),\dots,\pi(p_{x_r}),\pi(p_{y_r})$. By Lemma \ref{pi-lipschitz} the map $\pi$ is $1$-Lipschitz, and thus we have $\abs{\pi(p_{y_i})-\pi(p_{x_i})}\le\abs{y_i-x_i}$ for $i=0,\dots,r$.

\begin{figure}[ht]
    \centering
    \includegraphics[scale=1]{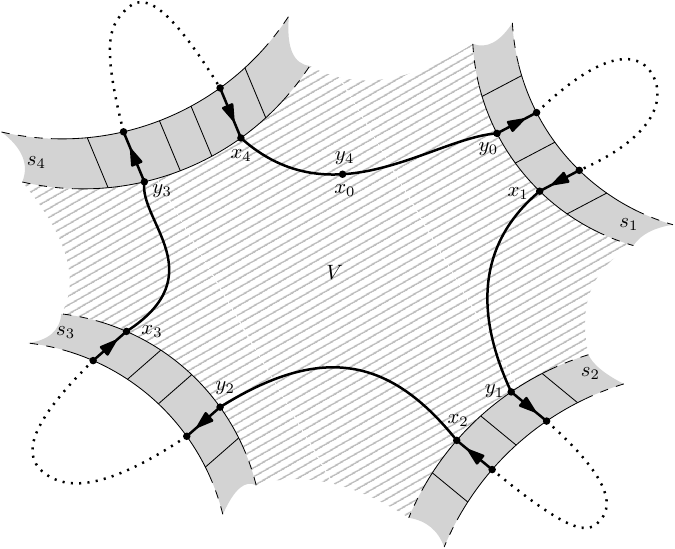}
    \caption{An example of a circuit $p$ entering $V$ four times.}
    \label{fig:circuit}
\end{figure}

Recall that we have a projection map $\tau:\ot X\rar T$ onto a tree $T$. The path $p=(p_0,\dots,p_k)$ induces a sequence of vertices $\tau(p_0),\dots,\tau(p_k)$ in the tree $T$, where $\tau(p_{j-1}),\tau(p_j)$ either coincide or are adjacent in $T$, for $j=1,\dots,k$. Fix $i\in\{1,\dots,r\}$ and observe that $\tau(p_{y_{i-1}})=\tau(p_{x_i})=\tau(V)$ and none of the vertices $\tau(p_{y_{i-1}+1}),\dots,\tau(p_{x_i-1})$ coincide with $\tau(V)$: since $T$ is a tree, this implies that $\tau(p_{y_{i-1}+1})=\tau(p_{x_i-1})$. It follows (see also Figure \ref{fig:circuit}) that the vertices $p_{y_{i-1}},p_{y_{i-1}+1},p_{x_{i}-1},p_{x_{i}}$ all belong to a common strip $s_i$ with one side glued on $V$, for $i=1,\dots,r$. We also observe that these are the only cases where the path $p$ crosses a strip with a side glued onto $V$.


If $s_i$ is glued onto $V$ along a $b$-line which is not parallel to $\lambda$, then $\pi(p_{x_{i}})-\pi(p_{y_{i-1}})=0$ since $p_{x_{i}},p_{y_{i-1}}$ belong to a same $b$-line which is sent to a constant value.
If $s$ is glued onto $V$ along a $b$-line $\lambda'$ parallel to $\lambda$, then there are two edges $l_n,l_m$ of $s_i$ with $\dpar l_n=p_{y_{i-1}+1}-p_{y_{i-1}}$ and $\dpar l_m=p_{x_{i}-1}-p_{x_{i}}$.
Since the map $\pi$ is a translation on the vertices of $\lambda'$, it follows that $\pi(p_{x_{i}})-\pi(p_{y_{i-1}})=m-n$, and notice that $m$ and $-n$ are two summands that appear in the definition of $\area_s(p)$.
For each strip $s$ glued to $V$ onto a $b$-line parallel to $\lambda$, each summand in the definition of $\area_s(p)$ appears exactly once when $i$ ranges from $1$ to $r$; this implies that
$$\sum\limits_{s\in S_\lambda}\area_s(p)\ =\ \sum\limits_{i=1}^{r}\pi(p_{x_{i}})-\pi(p_{y_{i-1}}).$$
It follows that
\begin{align*}
\left\lvert \sum\limits_{s\in S_\lambda}\area_s(p) \right\rvert\ &=\ \left\lvert \sum\limits_{i=1}^{r}\pi(p_{x_{i}})-\pi(p_{y_{i-1}}) \right\rvert\ =\ \left\lvert -\sum\limits_{i=0}^{r}\pi(p_{y_i})-\pi(p_{x_i}) \right\rvert\\
&\le\ \sum\limits_{i=0}^r\abs{\pi(p_{y_i})-\pi(p_{x_i})}\ \le\ \sum\limits_{i=0}^r\abs{y_i-x_i}\ =\ \len_V(p),
\end{align*}
where we used that $p_{y_r}=p_k=p_0=p_{x_0}$. The conclusion follows.
\end{proof}

Recall that we have a quotient map $\tau:\ot X\rar T$ that collapses each copy of $\ot Y$ to a point and each strip $\bR\times[0,1]$ to a segment, and that the quotient space $T$ is a tree. We now introduce a coloring, i.e., an equivalence relation, on the set of the edges of $T$.

Let $v$ be a vertex of $T$, corresponding to a copy $V$ of $\ot Y$. Let $e_1,e_2$ be edges of $T$ adjacent to the vertex $v$, corresponding to two strips $s_1,s_2\cong\bR\times[0,1]$ glued onto two $b$-lines $\lambda_1,\lambda_2$ in $V$. We say that $e_1,e_2$ are \textbf{$v$-parallel} if $\lambda_1,\lambda_2$ are parallel $b$-lines in $V$. Notice that being $v$-parallel is an equivalence relation on the set of edges of $T$ adjacent to $v$.

We now consider, on the set of edges of $T$, the equivalence relation generated by all the relations of being $v$-parallel for $v$ vertex of $T$; this gives us a \textbf{coloring} of the edges of $T$. This means that two edges $e,e'$ have the same color if and only if there is a sequence of edges $e=e_0,e_1,\dots,e_{k-1},e_k=e'$ in $T$ such that, for $i=0,\dots,k-1$, the edges $e_i,e_{i+1}$ have a common vertex $v_i$ and are $v_i$-parallel. In particular, since $T$ is a tree, the following holds: given two edges $e,e'$ of $T$ adjacent to a common vertex $v$, we have that $e,e'$ have the same color if and only if $e,e'$ are $v$-parallel.

\begin{mylemma}\label{coloring-trees}
Given the above coloring of the edges of $T$, there is a partial coloring of the vertices of $T$ (using the same colors) such that the following property holds: for each edge $e$ of $T$, the edge $e$ has exactly one endpoint of its same color.
\end{mylemma}
\begin{proof}
Fix a vertex $v_0$ of $T$ and let $S(v_0,n)$ be the set of vertices of $T$ which have distance exactly $n$ from $v_0$. We define by induction on $n$ a partial coloring on $S(v_0,n)$.

For $S(v_0,0)=\{v_0\}$ we just leave the vertex $v_0$ uncolored. Suppose we have defined the partial coloring on $S(v_0,n)$. Take a vertex $v\in S(v_0,n+1)$: since $T$ is a tree, then there is a unique edge $e$ connecting $v$ to a vertex $v'\in S(v_0,n)$. If $v'$ has the same color as $e$, then we leave $v$ uncolored; otherwise, we give $v$ the same color as $e$.

Since $S(v_0,n)$ for $n\in\bN$ form a partition of the set of vertices of $T$, this defines a partial coloring on the set of vertices of $T$. For how the partial coloring is defined, it is immediate to see that it has the desired property.
\end{proof}

Now we have a coloring of the edges and of a subset of vertices of $T$ with the following properties:
given two strips $s_1,s_2$ with a side glued onto a common copy $V$ of $\ot Y$, we have that $\tau(s_1),\tau(s_2)$ have the same color if and only if the two strips $s_1,s_2$ are glued onto parallel $b$-lines in $V$;
for each strip $s$ with the sides glued on two copies $V_1,V_2$ of $\ot Y$, we have that exactly one of $\tau(V_1),\tau(V_2)$ has the same color as $\tau(s)$.
We now use this coloring to prove the following fundamental proposition.

%

\begin{myprop}\label{prop:isoperimetric}
    Let $p = (p_0, \dots, p_k)$ be a circuit in $\ot{X}$.
    Then $\abs{\area(p)} \le \len(p)$.
\end{myprop}
\begin{proof}
Suppose $s$ is a strip of the form $\bR\times[0,1]$ in $\ot X$.
The edge $\tau(s)$ has a certain color, and exactly one of its endpoints has the same color: call $V(s)$ the $\tau$-preimage of that vertex; this means that $V(s)$ is a copy of $\ot Y$, and $\tau(V(s))$ is a vertex of $T$ with the same color as the edge $\tau(s)$.

Let now $V$ be a copy of $\ot Y$ in $\ot X$. Suppose that there is a strip $s$ with $V(s)=V$: then for every strip $s'$ we have that $V(s')=V$ if and only if $s'$ is glued on $V$ and $\tau(s')$ has the same color as $\tau(s)$; notice that, by definition, $\tau(s')$ has the same color as $\tau(s)$ if and only if they are glued onto two parallel $b$-lines in $V$. Thus we can apply Lemma \ref{area-length} and we have that
$$\left\lvert \sum\limits_{\substack{s\text{ with} \\ V(s)=V}}\area_s(p) \right\rvert\ \le\ \len_V(p),$$
where the sum on the left is finite since $p$ can only touch a finite number of strips. It follows that
$$\left\lvert\area(p)\right\rvert\ =\ \left\lvert\sum\limits_{\substack{s\text{ strip}}}\area_s(p)\right\rvert\ \le\  \sum\limits_{\substack{V\text{ copy}\\ \text{of }\ot Y}}\left\lvert\sum\limits_{\substack{s\text{ with} \\ V(s)=V}}\area_s(p)\right\rvert \ \le\ \sum\limits_{\substack{V\text{ copy}\\ \text{of }\ot Y}}\len_V(p)\ \le\ \len(p),$$
where the copies $V$ of $\ot Y$ for which there is no strip $s$ with $V(s)=V$ are meant to contribute with a zero to the sums (in particular the sums have only a finite number of non-zero summands, since the circuit $p$ touches a finite number of copies of $\ot Y$ and of strips by Remark \ref{rmk:finite}). The conclusion follows.
\end{proof}

\begin{myprop}
    The cohomology class $[\alpha_\bR] \in H^2_\mathrm{CW}(X;\bR) \cong H^2(G;\bR)$ is weakly bounded.
\end{myprop}
\begin{proof}
    Recall from Section \ref{sec:cohomology} that the class $[\alpha_\bR]$ is weakly bounded if and only if its corresponding class in the singular cohomology $H^2(X;\bR)$ lies in the kernel of the change of coefficient map $\iota^2:H^2(X;\bR) \to H_\Linfs^2(X;\bR) = H^2(X;\ell^\infty(G,\bR))$.
    Here, we prefer to work directly with cellular cochains;
    since the diagram
    \[
    \begin{tikzcd}[contains/.style = {draw=none,"\in" description}]
    	H^2_\mathrm{CW}(X;\mathbb{R}) \ar[d,"\cong"] \ar[r,"\iota^2_\mathrm{CW}"] & H_\mathrm{CW}^2(X;\ell^\infty(G,\bR))  \ar[d,"\cong"]\\
    	H^2(X;\mathbb{R}) \ar[r,"\iota^2"] & H^2(X;\ell^\infty(G,\mathbb{R}))
    \end{tikzcd}
    \]
    commutes, we are left to show that $[\alpha_\bR]$ lies in the kernel of $\iota^2_\mathrm{CW}:H_\mathrm{CW}^2(X;\bR)\to H_\mathrm{CW}^2(X;\ell^\infty(G,\bR))$.
    Since $X$ has finite $1$-skeleton, to prove this claim we can use the following result from \cite{Mil2021}, that characterizes the kernel of $\iota^2_\mathrm{CW}$ in terms of a linear isoperimetric inequality: $\iota^2_\mathrm{CW}([\alpha_\bR]) = 0$ if and only if there is a constant $L \in \bR$ such that $\abs{\ot{\alpha}_\bR(c)} \le L\norma{\partial c}_1$ for every $c \in C_2^\mathrm{CW}(\ot{X};\bR)$.
    We proceed to show this inequality with $L = 1$.
    
    Let $c \in C_2^\mathrm{CW}(\ot{X};\bR)$.
    The boundary of $c$ is a cellular $1$-cycle, and we express it as
    \[\partial c\ =\ \lambda_1 \ol{p_1} + \lambda_2 \ol{p_2} + \dots + \lambda_k \ol{p_k},\]
    where each $\ol{p_i}$ is the $1$-cycle associated to a circuit $p_i$, the coefficients $\lambda_i$ are non-negative real numbers, and no cancellation occurs in the sum, i.e., $\norma{\partial c}_1 = \lambda_1 \len(p_1) + \dots + \lambda_k \len(p_k)$.
    Every $1$-cycle can be expressed in this way, as it is easy to prove by induction on the number of summands appearing in the linear combination of $1$-cells defining the cycle.

    Since $\ot{X}$ is simply connected, for every $i \in \{1,\dots,k\}$ there is a $2$-chain $c_i \in C_2^\mathrm{CW}(\ot{X};\bR)$ such that $\partial c_i = \ol{p_i}$.
    By linearity of the boundary operator, we have $\partial c = \partial(\lambda_1 c_1 + \dots + \lambda_k c_k)$.
Recall that $\alpha_\bR$ is exact (by Remark \ref{rmk:alpha-exact}), and thus we have $\ot{\alpha}_\bR(c) = \ot{\alpha}_\bR(\lambda_1 c_1 + \dots + \lambda_k c_k)$.
    
    Applying in order the triangle inequality, Lemma \ref{alpha_area} and Proposition \ref{prop:isoperimetric} we obtain that
    \begin{align*}
        \abs{\ot{\alpha}_\bR(c)}\ &=\ \abs{\ot{\alpha}_\bR(\lambda_1 c_1 + \dots + \lambda_k c_k)}\\
        &\le\ \lambda_1 \abs{\ot{\alpha}_\bR(c_1)} + \dots + \lambda_k \abs{\ot{\alpha}_\bR(c_k)}\\
        &=\ \lambda_1 \abs{\area(p_1)} + \dots + \lambda_k \abs{\area(p_k)}\\
        &\le\ \lambda_1 \len(p_1) + \dots + \lambda_k \len(p_k)\\
        &=\ \norma{\partial c}_1,
    \end{align*}
    which is the inequality we wanted to prove.
\end{proof}

\section{A 2-dimensional CAT(0) model}\label{sec:cat0}

Let $G$ be the group constructed in Section \ref{sec:example}.
In this section we build a finite $2$-dimensional piecewise-Euclidean simplicial complex which is locally CAT(0) and whose fundamental group is isomorphic to $G$, thus finishing the proof of Theorem \ref{mainthm}.

Recall that the group $G$ is defined as the amalgamated product $G = H*_\gen{b}H$, where
\begin{equation}\label{eq:presentation_H1}
    H = \pres{a,b,t}{tbt^{-1} = b, tat^{-1} = ba}.
\end{equation}
By introducing two auxiliary letters, we can equivalently write
\begin{equation}\label{eq:presentation_H}
    H = \pres{a,b,t,x,y}{by = t, ax = y, yb = t, xa = t}.
\end{equation}
In fact, from the first two relations in (\ref{eq:presentation_H}) we get $y = b^{-1}t$ and $x = a^{-1}b^{-1}t$; by substituting these expressions in the last two relations we recover (equivalent forms of) the relations in (\ref{eq:presentation_H1}).

\begin{myprop}\label{prop:fcat0}
	The group $G$ constructed in Section \ref{sec:example} is of type F and CAT(0).
\end{myprop}
\begin{proof}
We start by constructing a locally CAT(0) $2$-dimensional model for $K(H,1)$, that we call $Y$.
Then, we obtain the desired model for $K(G,1)$ from two copies of $Y$ and a flat cylinder, by gluing the boundary components of the cylinder along closed geodesics representing the element $b \in H$ in the two copies of $Y$.

We build $Y$ as a CW complex with one $0$-cell, five $1$-cells and four triangular $2$-cells.
We denote the $0$-cell by $v$, and label the five $1$-cells with the letters $a$,$b$,$t$,$x$ and $y$.
The four $2$-cells are glued as shown in Figure \ref{fig:triangoli}.
The fundamental group of $Y$ is isomorphic to $H$, since the four triangles in Figure \ref{fig:triangoli} precisely encode the relations appearing in the presentation (\ref{eq:presentation_H}).

\begin{figure}[ht]
    \centering
    \includegraphics{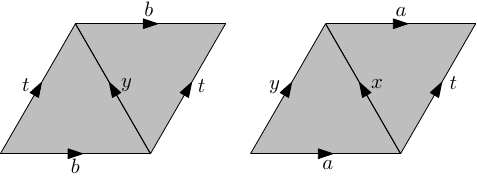}
    \caption{How to glue the $2$-cells of $Y$ to the $1$-skeleton.}
    \label{fig:triangoli}
\end{figure}

We obtain a simplicial complex by subdividing every $2$-cell of $Y$ as shown in Figure \ref{fig:subdivision} (here, we could consider the second barycentric subdivision, which is the standard way to obtain a simplicial complex from a $\Delta$-complex, but our choice makes the rest of the proof easier).
Then, we endow each $2$-simplex of the subdivision with the metric of a regular Euclidean triangle with side-length $1$, and consider the resulting path metric on $Y$.

\begin{figure}[ht]
    \centering
    \includegraphics{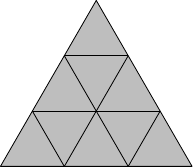}
    \caption{How to subdivide the $2$-cells of $Y$ to obtain a simplicial complex.}
    \label{fig:subdivision}
\end{figure}

We now check that $Y$, as a metric space, is locally CAT(0).
By \cite[Chapter II, Theorems 5.2 and 5.6]{BH1999}, $Y$ is locally CAT(0) if and only if its vertex links don't contain injective loops of length strictly less than $2\pi$.
The vertex links are $1$-dimensional simplicial complexes whose $1$-simplices all have length $\pi/3$, the latter being the amplitude of all angles of the $2$-simplices of $Y$.
For vertices distinct from $v$ (coming from the subdivision of a $1$-cell or $2$-cell), this conditions is very easily checked.
We now consider the link of $v$.

The $1$-simplices of the link of $v$ correspond to the angles of the triangles in Figure \ref{fig:triangoli}.
The resulting $1$-complex, drawn in Figure \ref{fig:vert_link}, has ten vertices and twelve edges. Every injective loop in it has at least six edges, and therefore its length is at least $2\pi$.
This implies that $Y$ is locally CAT(0).
In particular, the universal cover of $Y$ is CAT(0), hence contractible, and $Y$ is a model for $K(H,1)$.

\begin{figure}[ht]
    \centering
    \includegraphics{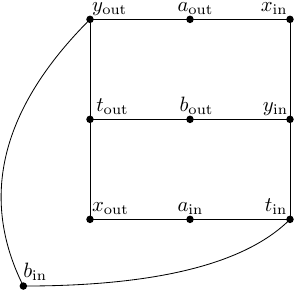}
    \caption{The link of $v$.
    It has ten vertices: every $1$-cell of $Y$ (before the subdivision) gives rise to two vertices, corresponding to its tail ($\null_\mathrm{out}$) and its head ($\null_\mathrm{in}$).}
    \label{fig:vert_link}
\end{figure}

Consider now the model for $K(G,1)$ obtained by taking two disjoint copies of $Y$ and gluing the boundary components of a cylinder to the $1$-cells labeled with $b$, one in each copy of $Y$.
We triangulate the cylinder as in Figure \ref{fig:triang_cyl}, and endow each $2$-simplex with the metric of a regular Euclidean triangle with side-length $1$.
The lower and upper sides are glued isometrically to the $1$-cell $b$ of the first and second copies of $Y$, respectively (recall that $b$ has been subdivided in three $1$-simplices).

\begin{figure}[ht]
    \centering
    \includegraphics{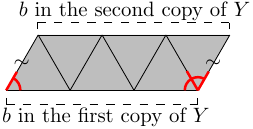}
    \caption{How to triangulate the cylinder. The edges and angles that are marked give rise, respectively, to two new vertices and three new edges in the link of $v$ in the first copy of $Y$.}
    \label{fig:triang_cyl}
\end{figure}

The link of $v$ changes as follows: two vertices and three edges are added, forming a path from $b_\mathrm{out}$ to $b_\mathrm{in}$ of length $\pi$.
After this operation, every loop still has length at least $2\pi$, since the minimal paths between $b_\mathrm{in}$ and $b_\mathrm{out}$ already had length at least $\pi$.
Hence, the resulting model for $K(G,1)$ is locally CAT(0), as desired.
\end{proof}

We use the fact that $G$ has a finite simplicial model to construct an aspherical counterexample to Gromov's conjecture, thus proving Corollary \ref{cor:false-conjecture}.
We start with the following observation.
\begin{mylemma}\label{lemma:retract}
	Let $r:G_2\to G_1$ be a group retraction.
	Suppose that $x \in H^2(G_1;\mathbb{Z})$ is weakly bounded and not bounded.
	Then the same is true for $r^*x \in H^2(G_2;\mathbb{Z})$.
\end{mylemma}
\begin{proof}
	It follows from the definitions that boundedness and weak boundedness are preserved by pull-backs.
	Therefore, $r^*x$ is weakly bounded.
	Let $i:G_1 \to G_2$ be an injective homomorphism such that $r \circ i: G_1 \to G_1$ is the identity of $G_1$.
	If $r^*x$ were bounded, then $i^*r^*x = x$ would be bounded.
\end{proof}
	
\begin{proof}[Proof of Corollary \ref{cor:false-conjecture}]
    Let $N$ be a compact smooth manifold with boundary which is homotopy equivalent to the finite simplicial model for $K(G,1)$ constructed in the proof of Proposition \ref{prop:fcat0}.
    In particular, $N$ is aspherical and its fundamental group is isomorphic to $G$.
    
    For example, such a manifold can be obtained by replacing $k$-dimensional simplices with $n$-dimensional $k$-handles, the result being a manifold of dimension $n$ (in our case, we have to take $n \ge 4$ to make the procedure work):
    one starts with disjoint $n$-discs (the $0$-handles) corresponding to the vertices of the simplicial complex; then, for every $1$-simplex, a $1$-handle is attached, joining the two $0$-handles corresponding to the endpoints of the simplex; finally, for every $2$-simplex $\sigma$, a $2$-handle is attached along a circle contained in the (boundary of the) union of the $0$- and $1$-handles corresponding to the faces of $\sigma$.
    More precisely, this union is diffeomorphic to $S^1\times D^{n-1}$, and the attaching circle is taken homotopic to $S^1\times\{*\}$.
    As long as $n \ge 4$, the attaching circles can be made disjoint by a transversality argument, and the procedure gives the desired manifold $N$.
    It follows from an application of \cite[Theorem 7.5.7]{Brown1968} that $N$ is homotopy equivalent to the original simplicial complex. 
    
    Now, we apply the Davis' reflection group trick described in \cite[\S 11.1]{Davis2008} to $(N,\partial N)$, obtaining a closed smooth manifold $M$ together with continuous maps $\iota:N \to M$ and $r:M \to N$ whose composition $r \circ \iota:N \to N$ is the identity.
    
    Let $[\alpha] \in H^2(G;\bZ) \cong H^2(\pi_1(N);\bZ)$ be the weakly bounded but not bounded class considered in Section \ref{sec:example}.
    By Lemma \ref{lemma:retract}, also $r^*[\alpha] \in H^2(\pi_1(M);\bZ)$ is weakly bounded and not bounded.
    By \cite[Corollary 20]{FS2020} it follows that $M$ (endowed with an arbitrary Riemannian metric) has a $2$-form $\omega$ that satisfies the two conditions in the statement.
\end{proof}

Belegradek proved in \cite{Belegradek2006} that for any closed aspherical (smooth or PL) $n$-manifold $M$ there is a closed aspherical $(n+1)$-manifold $M'$ in the same category such that $M'$ retracts onto $M$ and $\pi_1(M')$ is hyperbolic relative to $\pi_1(M)$.
By applying this result to the manifold $M$ of Corollary \ref{cor:false-conjecture}, we obtain another aspherical counterexample to Gromov's conjecture, with the additional property that its fundamental group is non-elementary relatively hyperbolic.
This also proves Corollary \ref{cor:rel-hyp}.

\bibliographystyle{fram_plain}
\bibliography{bibliography.bib}

\end{document}